\documentclass{article}
\usepackage[hmarginratio=1:1]{geometry}
\usepackage{titlesec}
\usepackage{pdfpages}
\titleformat{\section}  
  {\fontsize{14}{16}\bfseries} 
  {\thesection} 
  {1em} 
  {} 
  [] 

\titleformat{\subsection}  
  {\fontsize{12}{14}\bfseries} 
  {\thesubsection} 
  {1em} 
  {} 
  [] 

\usepackage[utf8]{inputenc}
\usepackage[utf8]{inputenc}
\usepackage{tikz, tikz-cd}
\usepackage{url, hyperref}
\usepackage{amsmath, amsthm, amssymb, stmaryrd, xcolor, enumitem}
\usepackage{graphicx}
\usepackage{setspace}
\usepackage{fancyref}
\usepackage{hyperref}
\hypersetup{colorlinks,linkcolor={blue},citecolor={blue},urlcolor={red}}
\usepackage{biblatex}
\usepackage{comment}
\usepackage{breqn}
\newtheorem{theorem}{Theorem}[section]
\usepackage{mathtools}
\usepackage{floatrow}
\usepackage{newunicodechar}

\theoremstyle{definition}
\newtheorem{definition}[theorem]{Definition}
\newtheorem{example}[theorem]{Example}
\newtheorem{remark}[theorem]{Remark}

\theoremstyle{plain}

\newtheorem{corollary}[theorem]{Corollary}
\newtheorem{lemma}[theorem]{Lemma}
\newtheorem{conjecture}[theorem]{Conjecture}
\newtheorem{conjecture/theorem}[theorem]{Conjecture/Theorem}

\newtheorem{proposition/definition}[theorem]{Proposition/Definition}
\newtheorem{tentative definition}[theorem]{Tentative Definition}

\newenvironment{explanation of the claim}{%
  \proof}{\endproof}

\newcommand{\gdrbmath}{\mathrm{G}_{\mathrm{dRB}}}
\newcommand{\gm}{$\mathbb{G}_{\mathrm{m}}$ }

\newcommand{\resgmmath}{\mathrm{Res}_{E/\mathbb{Q}}\mathbb{G}_{\mathrm{m}}}
\newcommand{\gmmath}{\mathbb{G}_{\mathrm{m}} }

\newcommand{\gdrbhmath}{\mathrm{G}_{\mathrm{dRB}}^{\mathrm{h}} }

\newcommand{\absgalois}{\mathrm{Gal}(\overline{\mathbb{Q}}/\mathbb{Q})}
\newcommand{\galoisL}{\mathrm{Gal}(L/\mathbb{Q})}

\newcommand{\qbar}{\overline{\mathbb{Q}}}

\newcommand{\betti}{\mathrm{H}^1(A,\mathbb{Q})}
\newcommand{\drba}{\mathrm{H}^1_{\mathrm{dRB}}(A,\mathbb{Q})}

\title{Weight of the De Rham-Betti Structures of Abelian Varieties}
\author{Zekun Ji}
\date{}
\begin{document}
\maketitle
\begin{abstract}
    In this note, we prove that for any abelian variety $A$ defined over $\qbar$, its de Rham-Betti (dRB) group $\mathrm{G}_{\mathrm{dRB}}(A)$ necessarily contains the group of homotheties in $\mathrm{GL}(\mathrm{H}^1_{\mathrm{B}}(A, \mathbb{Q}))$. Consequently, this rules out the existence of non-zero dRB classes in odd-degree cohomology groups of abelian varieties over $\qbar$. 
\end{abstract}
\tableofcontents
\section{Introduction}
Let $X$ be a smooth projective variety defined over $\qbar$. Grothendieck established the following comparison isomorphism, based on the geometry of $X$ (see \cite{grothendieck1966rham}): $$\rho_{m}: \mathrm{H}^{j}_{\mathrm{B}}(X,\mathbb{Q})\otimes_{\mathbb{Q}} \mathbb{C} \cong \mathrm{H}^{j}_{\mathrm{dR}}(X/\qbar) \otimes_{\qbar} \mathbb{C}.$$ 
Fix a $\mathbb{Q}$-basis of $\mathrm{H}^{j}_{\mathrm{B}}(X,\mathbb{Q})$ and fix a $\overline{\mathbb{Q}}$-basis of $\mathrm{H}^{j}_{\mathrm{dR}}(X/\overline{\mathbb{Q}})$. Then we call the entries of the matrix representing $\rho_{m}$ the periods on $X$. One facet of the Grothendieck period conjecture states that the polynomial relations with $\mathbb{Q}$-coefficients among periods on $X$ come from algebraic cycles.
\begin{conjecture}[\cite{andre2004introduction}, Conjecture 7.5.1.1]\label{wgpc}
  Let $X$ be a smooth projective variety defined over $\qbar$ and let $j$ be an integer with $0\leq j\leq \mathrm{dim}(X)$. If $\alpha\in\mathrm{H}_{\mathrm{B}}^{2j}(X,\mathbb{Q})$ satisfies $\rho_{m}(\alpha)\in(2\pi \sqrt{-1})^{-j}\mathrm{H}_{\mathrm{dR}}^{2j}(X/\qbar)$, then $\alpha$ is the image of a $\mathbb{Q}$-coefficient algebraic cycle on $X$ under the cycle class map.
\end{conjecture}

A stronger version of the Grothendieck period conjecture (see \cite[Conjecture 2.12]{bost2016some} for example) also predicts the following phenomenon in odd-degree cohomology groups.
\begin{conjecture}\label{oddgpc}
Let $X$ be a smooth projective variety defined over $\qbar$ and let $k$ be an odd integer. Then no non-zero element $\alpha\in\mathrm{H}_{\mathrm{B}}^{k}(X,\mathbb{Q})$ satisfies $\rho_{m}(\alpha)\in(2\pi\sqrt{-1})^{l}\mathrm{H}_{\mathrm{dR}}^{k}(X/\qbar)$ for any $l\in\mathbb{Z}$.   
\end{conjecture}
In \cite[Theorem 4.1]{bost2016some}, it is shown that for any smooth quasi-projective variety $X$ defined over $\qbar$, $\mathrm{H}_{\mathrm{dRB}}^{1}(X,\mathbb{Q}(l))$ does not support any dRB class for any $l\in\mathbb{Z}$. However, the case of higher odd-degree cohomology remains mysterious. In this note, we will focus on the case of abelian varieties and show the following result.
\begin{corollary}[=Corollary \ref{noodddrbclass}]
Let $A$ be an abelian variety defined over $\qbar$ and let $k$ be an odd integer. Then no non-zero element $\alpha\in\mathrm{H}_{\mathrm{B}}^{k}(A,\mathbb{Q})$ satisfies $\rho_{m}(\alpha)\in(2\pi\sqrt{-1})^{l}\mathrm{H}_{\mathrm{dR}}^{k}(A/\qbar)$ for any $l\in\mathbb{Z}$.   
\end{corollary}
Unlike the case of Hodge structures, where the notion of weight is built into the definitions, the general formalism of de Rham-Betti structures and groups lacks such a notion (see \cite[p.~3]{ksv2026rhambetticlassescoefficients} for a more detailed explanation). This is the main obstruction to directly extending \cite[Theorem 4.1]{bost2016some} to odd-degree cohomology groups of abelian varieties. The main goal of this note is to show the following property of the dRB groups of abelian varieties, which is a generalization of \cite[Proposition 4.7]{ji2025rhambettigroupstypeiv}.
\begin{theorem}\label{introthm1}[=Theorem \ref{thmgminside}]
    For any abelian variety $A$ defined over $\qbar$, its de Rham-Betti group $\gdrbmath(A)$ contains the group of homotheties in $\mathrm{GL}(\mathrm{H}_{\mathrm{B}}^{1}(A,\mathbb{Q}))$.
\end{theorem}

Combining this theorem with the Kuga-Satake correspondence (see, for example, \cite[Proposition 6.4.3]{andre1996shafarevich} and \cite[Theorem 6.1]{ksv2026rhambetticlassescoefficients}), we have the following observation.
\begin{corollary}[=Corollary \ref{corgminhk}]
Let $X$ be a hyper-Kähler variety defined over $\qbar$. Then the dRB group of $\mathrm{H}^{2}_{\mathrm{dRB}}(X,\mathbb{Q})$ contains the subgroup of homotheties in $\mathrm{GL}(\mathrm{H}_{\mathrm{B}}^{2}(X,\mathbb{Q}))$.
\end{corollary}
The proof of Theorem \ref{introthm1} uses the interaction between the de Rham-Betti group and the Mumford-Tate group of an abelian variety (see Section \ref{drbsection} for a summary of properties of these two groups) as well as properties of CM-fields.

Invoking Theorem \ref{introthm1}, we can also define the \textit{de Rham-Betti Hodge group} of an abelian variety (see Definition \ref{gdrbhdefn}), which we denote by $\gdrbhmath(A)$. It plays the role of the Hodge group in the de Rham-Betti setting. This paves the way for explicitly determining the dRB groups of some special cases of abelian varieties defined over $\qbar$ in a forthcoming paper by the author.

Another curious observation is that there are uncountably many isomorphism classes of one-dimensional dRB structures, essentially because there are uncountably many transcendental numbers in $\mathbb{C}$. This is starkly different from the theory of Hodge structures. Using Theorem \ref{introthm1}, we can translate the problem of determining one-dimensional dRB structures in the category $\langle\drba\rangle^{\otimes}$ into a group theoretic statement.

\begin{corollary}[=Corollary \ref{groupinterpret}]
Every one-dimensional dRB structure in $\langle\drba\rangle^{\otimes}$ is of the form $\mathbb{Q}_{\mathrm{dRB}}(k)$ for some $k\in\mathbb{Z}$ if and only if $\gmmath\cap\gdrbhmath(A)=\{\pm1\}\subset\mathrm{GL}(\betti)$.   
\end{corollary}

\subsection{Acknowledgments}
This is based on the first part of my PhD thesis (\cite{ji2025rhambettigroupstypeiv} or \cite{ji2026thesis}), supported by an NWO cluster grant with project number 613.009.153. I am very grateful to my advisor Mingmin Shen for introducing this topic to me and for helpful discussions. I have also benefited greatly from discussions with Charles Vial, Tobias Kreutz and Javier Frésan. I would also like to thank Charles Vial for providing useful comments on the draft. Tobias Kreutz passed away in August 2024. This article is deeply inspired by the insights of Tobias concerning the Tannakian category of de Rham-Betti structures.

\subsection{Notation and Conventions}
\begin{enumerate}
    \item We fix an algebraic closure $\qbar$ of $\mathbb{Q}$ and an embedding of fields $\qbar\xhookrightarrow{}\mathbb{C}$.
    \item Given an abelian variety $A$, we denote by $\mathrm{End}^{\circ}(A)$ its endomorphism algebra $\mathrm{End}(A)\otimes_{\mathbb{Z}}\mathbb{Q}$. 
    \item We fix a square root of $-1$ and denote it by $\sqrt{-1}$.
\end{enumerate}

\section{De Rham-Betti Groups and Mumford-Tate Groups}\label{drbsection}
In this section, we summarize some properties of de Rham-Betti groups and Mumford-Tate groups. For de Rham-Betti theory, more details can be found in \cite[Section 7.5.3]{andre2004introduction}, \cite{bost2016some} and \cite[Section 2]{ksv2026rhambetticlassescoefficients}. For the theory of Mumford-Tate groups, more details can be found in \cite{moonen2004introduction}.

A \textit{de Rham-Betti structure} is a triple of the form $(V_{\mathrm{dR}},V_{\mathrm{B}},\rho_{m})$ such that $V_{\mathrm{B}}$ is a finite-dimensional $\mathbb{Q}$-vector space, $V_{\mathrm{dR}}$ is a finite-dimensional $\qbar$-vector space, and $\rho_{m}$ is an isomorphism of $\mathbb{C}$-vector spaces $\rho_{m}: V_{\mathrm{B}} \otimes_{\mathbb{Q}}\mathbb{C} \cong V_{\mathrm{dR}} \otimes_{\qbar} \mathbb{C}$. An element $\alpha \in V_{\mathrm{B}}$ is a \textit{de Rham-Betti class} if $\rho_{m}(\alpha\otimes 1)\in V_{\mathrm{dR}}\otimes 1$. In the sequel, we sometimes abbreviate the adjective `de Rham-Betti' as `dRB'. Given a one-dimensional dRB structure $(V_{\mathrm{dR}},V_{\mathrm{B}},\rho_{m})$, if we can find a $\mathbb{Q}$-basis for $V_{\mathrm{B}}$ and a $\qbar$-basis for $V_{\mathrm{dR}}$ such that $\rho_{m}$ is scalar multiplication by $(2\pi \sqrt{-1})^{k}$, then we denote this dRB structure by $\mathbb{Q}_{\mathrm{dRB}}(k)$.

In \cite[Section 7.1.6]{andre2004introduction} and \cite[Section 2]{ksv2026rhambetticlassescoefficients}, it is explained that the category of all dRB structures, denoted by $\mathcal{C}_{\mathrm{dRB}}$, has a natural tensor structure. Together with the forgetful functor $\omega_{\mathrm{B}}$ to the category of $\mathbb{Q}$-vector spaces which remembers only $V_{\mathrm{B}}$ of a dRB structure, it forms a neutral Tannakian category.

Given a specific dRB structure $V$, we denote the Tannakian category generated by $V$ in $\mathcal{C}_{\mathrm{dRB}}$ by $\langle V\rangle^{\otimes}$. Then the \textit{de Rham-Betti group} of $V$ is defined as $\underline{\mathrm{Aut}}(\omega_{\mathrm{B}}|_{\langle V\rangle^{\otimes}})$ (see \cite[Page 20]{deligne2012tannakian} for a definition). We denote it by $\gdrbmath(V)$. It is an algebraic subgroup of $\mathrm{GL}(V_{\mathrm{B}})$. 
\begin{example}
Let $A$ be an abelian variety defined over $\qbar$. Then we have the famous Grothendieck comparison isomorphism $\rho_{m}: \mathrm{H}^{1}_{\mathrm{B}}(A,\mathbb{Q})\otimes_{\mathbb{Q}} \mathbb{C} \cong \mathrm{H}^{1}_{\mathrm{dR}}(A/\qbar) \otimes_{\qbar} \mathbb{C}$. We abbreviate the dRB structure $(\mathrm{H}^{1}_{\mathrm{dR}}(A/\qbar), \mathrm{H}^{1}_{\mathrm{B}}(A,\mathbb{Q}), \rho_{m})$ as $\mathrm{H}^{1}_{\mathrm{dRB}}(A,\mathbb{Q})$.
\end{example}
In the context of the above example, we denote $\gdrbmath(A):=\gdrbmath(\drba)$. The following theorem is fundamental for the subsequent analysis.
\begin{theorem}[\cite{ksv2026rhambetticlassescoefficients}, Theorem 2.6 and Theorem 5.12]\label{gdrbconnectedness}
For an arbitrary dRB structure $V$, the algebraic group $\gdrbmath(V)$ is connected. For an abelian variety $A$ defined over $\qbar$, $\gdrbmath(A)$ is a reductive group.
\end{theorem}

\begin{remark}
    The first property follows from the Tannakian formalism of dRB structures while the second is a corollary of the deep ``Analytic Subgroup Theorem" of Wüstholtz (\cite{wustholz1989algebraische}). See \cite[Section 5]{ksv2026rhambetticlassescoefficients} for a more detailed exposition.
\end{remark}

Another key property concerning the dRB group of an abelian variety is the following, which is also a corollary of the Analytic Subgroup Theorem.
\begin{theorem}[\cite{bost2016some}, Theorem 3.8]\label{mainfact}
Given an abelian variety $A$ defined over $\qbar$,
we have that $\mathrm{End}_{\mathrm{dRB}}(\betti)=\mathrm{End}_{\mathrm{Hdg}}(\betti)\cong\mathrm{End}^{\circ}(A)$.
\end{theorem}

The Mumford-Tate group is the fundamental symmetry group associated with a Hodge structure. We refer the reader to \cite{moonen2004introduction} for more details. Given a $\mathbb{Q}$-Hodge structure $V$ of weight $k$, i.e., $V \otimes_{\mathbb{Q}} \mathbb{C}=\bigoplus_{p+q=k}V^{p,q}$ such that $\overline{V^{p,q}}=V^{q,p}$, one can associate $V$ with an $\mathbb{R}$-representation of the Deligne torus $\mu:\mathbb{S}:=\mathrm{Res}_{\mathbb{C}/\mathbb{R}}\gmmath \rightarrow \mathrm{GL}(V_{\mathbb{R}})$ which is defined as follows. Given $z \in \mathbb{S}(\mathbb{R})=\mathbb{C}^{\times}$, its image under $\mu$ acts via scalar multiplication by $z^{-p}\overline{z}^{-q}$ on the summand $V^{p,q}$.

The \textit{Mumford-Tate group} $\mathrm{MT}(V)$ of a $\mathbb{Q}$-Hodge structure $V$ is the smallest $\mathbb{Q}$-algebraic subgroup of $\mathrm{GL}(V)$ whose set of $\mathbb{R}$-points contains the image of $\mu$. The \textit{Hodge group} of $V$, denoted by $\mathrm{Hdg}(V)$, is the smallest $\mathbb{Q}$-algebraic subgroup of $\mathrm{GL}(V)$ whose set of $\mathbb{R}$-points contains the image of the set $\{z \in \mathbb{C}^{\times}|z\overline{z}=1\}$ under $\mu$. These definitions yield the following well-known lemma.
\begin{lemma}\label{hodgetimesgm}
Let $V$ be a Hodge structure of weight $k$ where $k\neq0$. Then $\mathrm{MT}(V)$ contains the subgroup of homotheties in $\mathrm{GL}(V)$. Moreover, there is an isogeny $\Phi: \gmmath \times \mathrm{Hdg}(V) \rightarrow \mathrm{MT}(V)$ which takes $(t,g) \in \gmmath \times \mathrm{Hdg}(V)$ to $t^{-1}g \in \mathrm{MT}(V)$.
\end{lemma}
One can also define the Mumford-Tate group of a Hodge structure $V$ in the Tannakian formalism. By \cite{deligne2012tannakian}, the category of $\mathbb{Q}$-Hodge structures together with the forgetful functor $\omega$, which remembers only the underlying $\mathbb{Q}$-vector space, forms a neutral Tannakian category. Then the Mumford-Tate group of $V$ is in fact equal to $\underline{\mathrm{Aut}}(\omega|_{\langle V\rangle^{\otimes}})$. The following theorem from \cite{ksv2026rhambetticlassescoefficients} establishes a relation between the de Rham-Betti group and the Mumford-Tate group of an abelian variety defined over $\qbar$. 
\begin{theorem}[\cite{ksv2026rhambetticlassescoefficients}, Theorem 5.12]\label{keyinclusion}
    For an abelian variety $A$ defined over $\qbar$, we have that $\gdrbmath(A)\xhookrightarrow{}\mathrm{MT}(A)$ as subgroups of $\mathrm{GL}(\betti)$.
\end{theorem}
The above observation can be refined as follows.
\begin{lemma}\label{centreincentre}
    For any abelian variety $A$ defined over $\qbar$, the centre of its de Rham-Betti group $\mathrm{Z}(\gdrbmath(A))$ is a subgroup of $\mathrm{Z}(\mathrm{MT}(A))$, i.e., the centre of its Mumford-Tate group. Also, $\mathrm{Z}(\gdrbmath(A))$ naturally lies in $\mathrm{Z}(\mathrm{End}_{\mathrm{Hdg}}(\betti))$, i.e., the centre of the endomorphism algebra of $A$.
\end{lemma}
\begin{proof}
For any algebraic group $G\subset\mathrm{GL}(V)$, we have by definition $\mathrm{Z}(G)=\mathrm{End}_{G}(V) \cap G$. Recall that we have the natural inclusion of algebraic groups $\gdrbmath(A) \xhookrightarrow{} \mathrm{MT}(A)$ by Theorem \ref{keyinclusion} and by Theorem \ref{mainfact} we have that $\mathrm{End}_{\mathrm{Hdg}}(\mathrm{H}^1(A,\mathbb{Q})) = \mathrm{End}_{\mathrm{dRB}}(\mathrm{H}^1(A,\mathbb{Q}))$. Thus we obtain the inclusion of $\mathrm{Z}(\gdrbmath(A))$ in $\mathrm{Z}(\mathrm{MT}(A))$ as a subgroup. For the second inclusion, note that $\mathrm{Z}(\gdrbmath(A))$ commutes with every element of $\gdrbmath(A)$ which implies that it falls inside $\mathrm{End}_{\mathrm{dRB}}(\betti)$. Furthermore as a subgroup of $\gdrbmath(A)$, $\mathrm{Z}(\gdrbmath(A))$ commutes with every element of $\mathrm{End}_{\mathrm{dRB}}(\betti)$. We thus conclude that $\mathrm{Z}(\gdrbmath(A))\subset\mathrm{Z}(\mathrm{End}_{\mathrm{dRB}}(\betti))\cong\mathrm{Z}(\mathrm{End}^{\circ}(A))$.
\end{proof}
To relate the centre of dRB groups and $\gmmath$, we have the following lemma, which is a slight strengthening of \cite[Theorem 5.12]{ksv2026rhambetticlassescoefficients}.
\begin{lemma}\label{surjectiontogmcompatible}
For an abelian variety $A$ defined over $\qbar$, there exist surjective algebraic group homomorphisms from $\mathrm{Z}(\mathrm{MT}(A))$ and $\mathrm{Z}(\gdrbmath(A))$ to $\gmmath$ and moreover, they are compatible with each other. In other words, the following diagram commutes.
$$\begin{tikzcd}[column sep=small, row sep=small]
\mathrm{Z}(\gdrbmath(A)) \arrow[r, two heads] \arrow[d, hook] & \gmmath \\
\mathrm{Z}(\mathrm{MT}(A)) \arrow[ru, two heads]              &        
\end{tikzcd}$$
\end{lemma}
\begin{proof}
Let $g=\mathrm{dim}(A)$. Then $\wedge^{2g}\mathrm{H}^1(A,\mathbb{Q})$ is a weight $2g$ Hodge structure of dimension 1. Hence $\mathrm{MT}(\wedge^{2g}\mathrm{H}^1(A,\mathbb{Q}))=\gmmath$. On the other hand, elements in the one-dimensional dRB structure $\mathrm{H}_{\mathrm{dRB}}^{2g}(A,\mathbb{Q}) \otimes \mathbb{Q}_{\mathrm{dRB}}(g)$ are algebraic classes and hence are dRB classes. Therefore we obtain $\wedge^{2g}\mathrm{H}_{\mathrm{dRB}}^1(A,\mathbb{Q})=\mathrm{H}_{\mathrm{dRB}}^{2g}(A,\mathbb{Q}) \cong \mathbb{Q}_{\mathrm{dRB}}(-g)$. Due to the well-known fact that $\pi$ is a transcendental number, this implies that $\gdrbmath(\wedge^{2g}\mathrm{H}_{\mathrm{dRB}}^1(A,\mathbb{Q}))=\gmmath$. By \cite[Proposition 2.21]{deligne2012tannakian}, the inclusion of neutral Tannakian categories $\langle\wedge^{2g}\mathrm{H}^1(A,\mathbb{Q})\rangle^{\otimes} \xhookrightarrow{} \langle \mathrm{H}^1(A,\mathbb{Q})\rangle^{\otimes}$ and $\langle\wedge^{2g}\mathrm{H}^1_{\mathrm{dRB}}(A,\mathbb{Q})\rangle^{\otimes} \xhookrightarrow{} \langle \mathrm{H}_{\mathrm{dRB}}^1(A,\mathbb{Q})\rangle^{\otimes}$ induces surjective algebraic group homomorphisms $\mathrm{MT}(A) \twoheadrightarrow \gmmath$ and $\gdrbmath(A) \twoheadrightarrow \gmmath$. Upon close inspection, these two homomorphisms are the restrictions of the usual determinant maps on $\mathrm{GL}(\betti)$, hence are compatible with each other. Furthermore, both $\mathrm{MT}(A)$ and $\gdrbmath(A)$ are reductive groups. Then invoking Lemma \ref{surjectiontotorus} stated below, the surjective homomorphisms $\mathrm{Z}(\mathrm{MT}(A)) \twoheadrightarrow \gmmath$ and $\mathrm{Z}(\gdrbmath(A)) \twoheadrightarrow \gmmath$ hold because $\gdrbmath(A)$ and $\mathrm{MT}(A)$ are reductive. Finally we use Lemma \ref{centreincentre} to show the commutativity of the diagram.
\end{proof}

\begin{lemma}\label{surjectiontotorus}
    Suppose a reductive group $G$ admits a surjective homomorphism $f: G\rightarrow H$, where $H$ is a connected commutative algebraic group. Assume both $G$ and $H$ are defined over a field of characteristic zero. Then the restriction of $f$ to $\mathrm{Z}(G)$ is also a surjective homomorphism of algebraic groups.
\end{lemma}
\begin{proof}
Since $f$ is a surjective homomorphism, it follows that $\mathrm{Lie}(f): L:=\mathrm{Lie}(G) \twoheadrightarrow \mathrm{Lie}(H)$. Note that $L$ is a reductive Lie algebra. It then follows that $L=\mathrm{Z}(L)\oplus L^{\mathrm{ss}}=\mathrm{Lie}(\mathrm{Z}(G))\oplus [L,L]$. Moreover because $H$ is a commutative algebraic group and $[L,L]$ is a semisimple Lie algebra, the image of $\mathrm{Lie}(f)|_{[L,L]}$ inside $\mathrm{Lie}(H)$ is $\{0\}$. We then obtain that $\mathrm{Lie}(f)|_{\mathrm{Lie}(\mathrm{Z}(G))}=\mathrm{Lie}(f|_{\mathrm{Z}(G)})$ is a surjective morphism of Lie algebras. By Proposition 3.25 of \cite{milne2011algebraic}, we have that $H^{\circ}\subset f(\mathrm{Z}(G))$. Since $H$ is assumed to be a connected algebraic group, we obtain that $f|_{\mathrm{Z}(G)}:\mathrm{Z}(G) \twoheadrightarrow H$.
\end{proof}

\section{Properties of Subtori of \texorpdfstring{$\mathrm{U}_{E}$}{U\_E}}

The identity component of the centre of the de Rham-Betti group of an abelian variety over $\qbar$ is an algebraic torus. We now summarize some basic properties of algebraic tori as well as establish notations. The reader can find more details in \cite[Chapter 14]{milne2014algebraic}. Recall that an algebraic group $T$ defined over $\mathbb{Q}$ is a (connected) \textit{$\mathbb{Q}$-algebraic torus} if $T_{\qbar}$ is isomorphic to a finite product of $\mathbb{G}_{m,\qbar}$. The \textit{character group} of a $\mathbb{Q}$-algebraic torus is $\mathrm{Hom}(T_{\qbar},\mathbb{G}_{m,\qbar})$. We denote it by $X^{*}(T)$. When no confusion arises, we sometimes abbreviate a \textit{$\mathbb{Q}$-algebraic torus} as an \textit{algebraic torus}. The finite-rank $\mathbb{Z}$-module $X^{*}(T)$ admits an action of the absolute Galois group $\mathrm{Gal}(\qbar/\mathbb{Q})$ as follows: an element $g \in \mathrm{Gal}(\qbar/\mathbb{Q})$ induces homomorphisms of algebraic groups $g_{T}: T_{\qbar} \rightarrow T_{\qbar}$ and $g_{0}: \mathbb{G}_{m,\qbar} \rightarrow \mathbb{G}_{m,\qbar}$.  Then for a character $f \in X^{*}(T)$, we let \begin{equation}\label{galoisactiononchar}g\circ f=g_{0} \circ f \circ g_{T}^{-1}\in X^{*}(T).\end{equation} This is a continuous group action on the left where we equip $\mathrm{Gal}(\qbar/\mathbb{Q})$ with the Krull topology and $X^{*}(T)$ with the discrete topology. 

\begin{theorem}[\cite{milne2014algebraic}, Theorem 14.17] \label{chartorus}
 Taking the character group defines a contravariant functor $X^{*}$ from the category of $\mathbb{Q}$-algebraic tori to the category of finite rank free $\mathbb{Z}$-modules equipped with a continuous $\mathrm{Gal}(\overline{\mathbb{Q}}/\mathbb{Q})$-action. This is an equivalence of categories. Moreover, under this equivalence, exact sequences of algebraic tori correspond to exact sequences of $\mathrm{Gal}(\overline{\mathbb{Q}}/\mathbb{Q})$-modules.
     
\end{theorem}
An essential example of a $\mathbb{Q}$-algebraic torus is $\mathrm{Res}_{E/\mathbb{Q}}\gmmath$, where $E$ is a number field. We give a description of the $\mathrm{Gal}(\qbar/\mathbb{Q})$ structure of $X^{*}(\resgmmath)$. More details can be found in \cite[Chapter 14]{milne2014algebraic}.

\begin{example} \label{firstexample}
Recall that $\mathrm{Res}_{E/\mathbb{Q}}\gmmath$ is a $\mathbb{Q}$-algebraic torus such that for a $\mathbb{Q}$-algebra $R$, we have that $\mathrm{Res}_{E/\mathbb{Q}}\gmmath(R)=(E\otimes_{\mathbb{Q}}R)^{\times}$. In particular, we have that $\mathrm{Res}_{E/\mathbb{Q}}\gmmath(\qbar)=\prod_{\mathrm{Hom}(E,\overline{\mathbb{Q}})}\qbar^{\times}$.
Then we have an identification $$X^{*}(\resgmmath)\cong \mathbb{Z}^{\mathrm{Hom}(E,\overline{\mathbb{Q}})}$$ as $\absgalois$-modules. The free $\mathbb{Z}$-module $\mathbb{Z}^{\mathrm{Hom}(E,\overline{\mathbb{Q}})}$ admits a basis labeled by $\mathrm{Hom}(E,\overline{\mathbb{Q}})$. Under this identification, an element $(n_{\sigma})_{\sigma\in \mathrm{Hom}(E,\qbar)}\in \mathbb{Z}^{\mathrm{Hom}(E,\overline{\mathbb{Q}})}$ corresponds to a map $$f:\resgmmath(\qbar)\rightarrow\gmmath(\qbar),(z_{\sigma})_{\sigma\in \mathrm{Hom}(E,\qbar)}\mapsto\prod_{\sigma\in \mathrm{Hom}(E,\qbar)}z_{\sigma}^{n_{\sigma}}\in\qbar^{\times}.$$ By formula (\ref{galoisactiononchar}), the left group action by $\absgalois$ on $\mathbb{Z}^{\mathrm{Hom}(E,\overline{\mathbb{Q}})}$ corresponds to the $\absgalois$-structure on $X^{*}(\resgmmath)$. Let $L$ be the Galois closure of $E/\mathbb{Q}$ in $\overline{\mathbb{Q}}$.
Note that $(\mathrm{Res}_{E/\mathbb{Q}}\mathbb{G}_{m})_{L}$ already splits as a product of $\mathbb{G}_{\mathrm{m},L}$. Hence the action of $\absgalois$ on  $X^{*}(\resgmmath)$ in fact descends to the action of $\galoisL$ on $\mathbb{Z}^{\mathrm{Hom}(E,L)}$. 
\end{example}
Recall that a number field $E$ is a \textit{CM-field} if there exists a unique nontrivial field automorphism $\tau \in \mathrm{Aut}(E/\mathbb{Q})$ such that for any embedding $i: E \rightarrow \mathbb{C}$ and any $e\in E$, we have $\overline{i(e)}=i(\tau e)$. We will call $\tau$ the complex conjugation on $E$.

\begin{lemma}[\cite{milne2006complex}, Proposition 1.1.4 and Corollary 1.1.5] \label{comjugationcommuting}
For a CM-field $E$, its complex conjugation commutes with any other field automorphism of $E$. A finite composites of CM-subfields of a field is CM, and the Galois closure $L$ of a CM-field $E/\mathbb{Q}$ in $\qbar$ is also a CM-field. 
\end{lemma}

\begin{example}\label{uexample}
Let $E$ be a CM-field. Then $\mathrm{U}_E$ is an algebraic subtorus of $\resgmmath$ such that for a $\mathbb{Q}$-algebra $R$, we have that $\mathrm{U}_E(R)=\{x \in (E\otimes_{\mathbb{Q}}R)^{\times}|x(\tau\circ x)=1\}$, where the complex conjugation $\tau$ on $E$ extends $\mathbb{Q}$-linearly to $E\otimes_{\mathbb{Q}}R$. Then it follows that $$\mathrm{U}_E(\qbar)=\{(z_{\sigma})_{\sigma \in \mathrm{Hom}(E,\qbar)}\in\prod_{\mathrm{Hom}(E,\overline{\mathbb{Q}})}\qbar^{\times}|z_{\sigma}z_{\overline{\sigma}}=1\} \subset \resgmmath(\qbar).$$ 
Using the explicit description from Example \ref{firstexample}, one can see that for any element $f \in X^{*}(\resgmmath)$ and for any element $z\in \mathrm{U}_{E}(\qbar)$, we have that $f(z)=(\tau\circ f)(z)^{-1}$.  Using Lemma \ref{comjugationcommuting}, one can check that the complex conjugation in $\absgalois$ acts as scalar multiplication by $-1$ on $X^{*}(\mathrm{U}_E)$.   \end{example}
 
\begin{lemma}\label{nosurjectiontogm}
Let $E_1,\dots, E_{n}$ be CM-fields. Then no $\mathbb{Q}$-algebraic subgroup of $\mathrm{U}_{E_1}\times\cdots\times\mathrm{U}_{E_n}$ admits a surjective homomorphism to $\gmmath$.
\end{lemma}

\begin{proof}
Let $E$ be the composite field of $E_1,\dots,E_{n}$ in $\qbar$, which is a CM field by Lemma \ref{comjugationcommuting}. Denote by $L$ the Galois closure of $E/\mathbb{Q}$ in $\qbar$, which is a Galois CM-field. Since $\gmmath$ is a connected algebraic group, it suffices to show that no connected algebraic subgroup of $\mathrm{U}_{E_1}\times\cdots\times\mathrm{U}_{E_n}$ admits a surjection to $\gmmath$. Let $T$ be an arbitrary $\mathbb{Q}$-algebraic subtorus of $\mathrm{U}_{E_1}\times\cdots\times\mathrm{U}_{E_n}$. Then by Theorem \ref{chartorus}, there is a $\mathrm{Gal}(L/\mathbb{Q})$-equivariant surjection of $\mathbb{Z}$-modules: \begin{equation}\label{chartandcharue}X^{*}(\mathrm{U}_{E_1})\oplus\cdots\oplus X^{*}(\mathrm{U}_{E_{n}}) \xrightarrow{f} X^{*}(T).\end{equation} Denote by $\tau$ the complex conjugation of the CM-field $L$. For any $m \in X^{*}(\mathrm{U}_{E_1})\oplus\cdots\oplus X^{*}(\mathrm{U}_{E_{n}})$, by Example \ref{uexample}, we have that $\tau\circ m=-m$. Then formula (\ref{chartandcharue}) implies that $$\tau \circ f(m)=f(\tau \circ m)=f(-m)=-f(m).$$ Because $f$ is surjective, $\tau$ acts as scalar multiplication by $-1$ on $X^{*}(T)$ as well. Suppose $T$ admits a surjective homomorphism to $\gmmath$. Then by Theorem \ref{chartorus} we would have a $\mathrm{Gal}(L/\mathbb{Q})$-equivariant inclusion $X^{*}(\gmmath) \xhookrightarrow{i} X^{*}(T)$. Note that \gm is already split over $\mathbb{Q}$, hence $\tau \in \mathrm{Gal}(L/\mathbb{Q})$ acts as identity on $X^{*}(\gmmath)$. Therefore, for any $a \in X^{*}(\gmmath)$, we would have that $i(a)=i(\tau \circ a)=\tau \circ i(a)=-i(a)$. Since $X^{*}(T)$ is a free $\mathbb{Z}$-module and $i$ is injective, this gives a contradiction. 
\end{proof}

\begin{corollary} \label{gminside}
Suppose $G$ is an algebraic group over $\mathbb{Q}$ such that no algebraic subgroup of $G$ admits a surjective homomorphism to $\gmmath.$ Given a $\mathbb{Q}$-algebraic subgroup $T \subset \gmmath \times G$ such that $\mathrm{pr}_1: T \rightarrow \gmmath$ is a surjective homomorphism, we have that $T \cap \gmmath \times \{\mathrm{id}_{G}\}=\gmmath \times \{\mathrm{id}_{G}\}$.
\end{corollary}

\begin{proof}
Denote the image of $T$ under $\mathrm{pr}_2: T \rightarrow G$ by $T'$, which is an algebraic subgroup of $G$. Then we have $$T \xhookrightarrow{} \gmmath \times T'$$ where both $\mathrm{pr}_1$ and $\mathrm{pr}_2$, when restricted to $T$, are surjective homomorphisms. Denote $T \cap (\gmmath \times \{\mathrm{id}_{T'}\})$ by $N$ and $T \cap \{\mathrm{id}_{\gmmath}\} \times T'$ by $N'$. By Goursat's Lemma \ref{goursat} (see \cite[Proposition 2.16]{lewis1999survey} for example) stated below, we obtain an isomorphism of algebraic groups $T'/N'\cong \gmmath/N$, which induces a surjective homomorphism of $\mathbb{Q}$-algebraic groups $T' \twoheadrightarrow \gmmath/N$. Note that $\gmmath/N$ is either isomorphic to $\gmmath$ or $\{\mathrm{id}\}$. However $T'$ being an algebraic subgroup of $G$, cannot admit a surjective homomorphism to $\gmmath$ by the assumption. Hence $\gmmath/N$ is equal to $\{\mathrm{id}\}$. Therefore we have that $N=T \cap \gmmath \times \{\mathrm{id}_{T'}\}=\gmmath \times \{\mathrm{id}_{T'}\}$.
\end{proof}
\begin{lemma}[Goursat's lemma] \label{goursat}
Let $G_1$ and $G_2$ be two algebraic groups, and $H$ be an algebraic subgroup of $G_1 \times G_2$ whose projections to both $G_1$ and $G_2$ are surjective. Denote $H \cap G_1 \times \{\mathrm{id}_{G_2}\}$ by $N_1$ and $H \cap \{\mathrm{id}_{G_1}\}\times G_2$ by $N_2$. Then $N_1$ is a normal algebraic subgroup of $G_1$ and $N_2$ is a normal algebraic subgroup of $G_2$. Moreover, the natural image of $H$ in $G_1/N_1 \times G_2/N_2$ is the graph of an algebraic group isomorphism between $G_1/N_1$ and $G_2/N_2$.

\end{lemma}

\section{Proof of the Theorem}
Certain properties of simple abelian varieties of type IV are needed. Recall that an abelian variety $A$ is of \textit{type IV} if its endomorphism algebra $R = \mathrm{End}^\circ(A)$ is a type IV division algebra (see \cite[Ch.~19]{mumford1970abelian} for the precise definition); in particular, its centre $E$ is a CM-field. Let $V = \betti$ and fix a polarization form $\phi$. Then the associated Rosati involution $\dag$ on $R$ satisfies $\phi(f(v),w) = \phi(v, f^\dag(w))$. Crucially, for abelian varieties of type IV, the restriction of $\dag$ to $E$ coincides with complex conjugation on the CM-field $E$ \cite[pp.~201-202]{mumford1970abelian}.

\begin{lemma}\label{hodgecenter}
Assuming the above setup and notation, we have $\mathrm{Z}(\mathrm{Hdg}(A))\subset\mathrm{U}_{E}$.
\end{lemma}
\begin{proof}
This is well known. We include the proof for completeness of exposition. Note that
$\mathrm{End}_{\mathrm{Hdg}(A)}(V)=\mathrm{End}_{\mathrm{MT}(A)}(V)=R.$
By definition we have
$\mathrm{Z}(G)=\mathrm{End}_{G}(V) \cap G$ for any algebraic group $G\subset\mathrm{GL}(V)$. Thus $\mathrm{Z}(\mathrm{Hdg}(A))\subset R$. Moreover, any element in $R$ commutes with $\mathrm{Hdg}(A)$, hence $\mathrm{Z}(\mathrm{Hdg}(A))$ is naturally a subgroup of $\mathrm{Z}(R)^{\times}=E^{\times}$.
Now note that for any element $g\in \mathrm{Z}(\mathrm{Hdg}(A))$ and any $v,v'\in V$, the polarization form $\phi$ satisfies $$\phi(gv,gv')=\phi(v,v')=\phi(v,\overline{g}gv')$$ because $\phi$ is preserved by $\mathrm{Hdg}(A)$ and the Rosati involution equals the complex conjugation on $E=\mathrm{Z}(R)$. Since $\phi$ is a non-degenerate form, we deduce that $\overline{g}g=1$. 
\end{proof}
\begin{theorem}\label{thmgminside}
    For any abelian variety $A$ defined over $\qbar$, its de Rham-Betti group $\gdrbmath(A)$ contains the group of homotheties in $\mathrm{GL}(\betti)$. 
\end{theorem}

\begin{proof}
By the Albert classification theorem (for example see \cite[Chapter 19]{mumford1970abelian}), we have that $A$ is isogenous to $A_1\times A_2 \times \cdots \times A_{m}$ ($m\geq 1$) where each $A_{i}$ is a simple abelian variety defined over $\qbar$. Moreover, each $A_{i}$ has endomorphism type I, II, III or IV. By the Tannakian formalism, we obtain the following diagram. $$\begin{tikzcd}[column sep=small, row sep=small]
\mathrm{MT}(A) \arrow[rr, "\iota", hook] &        & \mathrm{MT}(A_1)\times\cdots\times\mathrm{MT}(A_{m})\arrow[ld, two heads] \arrow[d, "\mathrm{pr}_{m}"', two heads] \arrow[lld, "\mathrm{pr}_1"', two heads] \\
\mathrm{MT}(A_1)                         & \cdots & \mathrm{MT}(A_{m})                                                                                                                                                                                                    
\end{tikzcd}$$ where each $\tilde{\mathrm{pr}_{i}}:=\mathrm{pr}_{i}\circ\iota$ is a surjective homomorphism of algebraic groups. We thus deduce the following relation among centres of the Hodge groups 
$$\mathrm{Z}(\mathrm{Hdg}(A))\xhookrightarrow{}\mathrm{Z}(\mathrm{Hdg}(A_1))\times\cdots\times\mathrm{Z}(\mathrm{Hdg}(A_{m})).$$
For each $i\in\{1,\dots,m\}$, if $A_{i}$ is of type I, II or III, then by \cite[Lemma 1.4]{tankeev1978algebraic} (or \cite[Proposition 1.24]{moonen2004introduction}), we have that $\mathrm{Hdg}(A_{i})$ is a semisimple algebraic group and $\mathrm{Z}(\mathrm{Hdg}(A_{i}))^{\circ}$ is trivial. If $A_{i}$ is of type IV, then by Lemma \ref{hodgecenter}, $\mathrm{Z}(\mathrm{Hdg}(A_{i}))\subset\mathrm{U}_{E_{i}}$ for the CM field $E_{i}$. Hence if there are no type IV factors among $A_1,\dots,A_{m}$, $\mathrm{Z}(\mathrm{Hdg}(A))^{\circ}$ is trivial. Otherwise, assuming without loss of generality that $A_1,\dots, A_{n}$ are the factors of type IV, then $\mathrm{Z}(\mathrm{Hdg}(A))^{\circ}$ is a connected subgroup of $\mathrm{U}_{E_1}\times\cdots\times\mathrm{U}_{E_{n}}$. 
By Lemma \ref{surjectiontogmcompatible}, we have the following commutative diagram where $\Phi(t,g) = t^{-1}g$ for $(t,g) \in \gmmath \times \mathrm{Z}(\mathrm{Hdg}(A))$.
\begin{equation}\label{typeivkeydiagram}
\begin{tikzcd}[column sep=1.5em, row sep=small]
\Phi^{-1}(\mathrm{Z}(\gdrbmath(A)))^{\circ} \arrow[d, hook] \arrow[rr, two heads] & & \mathrm{Z}(\gdrbmath(A))^{\circ} \arrow[d, hook] \arrow[rr, "\mathrm{det}", two heads] & & \gmmath \arrow[d, equal] \\
\gmmath \times \mathrm{Z}(\mathrm{Hdg}(A))^{\circ}\arrow[rr, "\Phi", two heads] & & \mathrm{Z}(\mathrm{MT}(A))^{\circ}  \arrow[rr, "\mathrm{det}", two heads] & & \gmmath   
\end{tikzcd}
\end{equation} 
Denote by $\mathrm{pr}_1$ the projection from $\gmmath \times \mathrm{Z}(\mathrm{Hdg}(A))^{\circ}$ to $\gmmath$. Then we claim that $\mathrm{pr}_1|_{\Phi^{-1}(\mathrm{Z}(\gdrbmath(A)))^{\circ}}$ is a surjective homomorphism. If not, then its image is trivial, which implies that $\Phi^{-1}(\mathrm{Z}(\gdrbmath(A)))^{\circ}$ lies in $\{\mathrm{id}\}\times\mathrm{Z}(\mathrm{Hdg}(A))^{\circ}$. However, the first row of diagram (\ref{typeivkeydiagram}) implies that $\Phi^{-1}(\mathrm{Z}(\gdrbmath(A)))^{\circ}$ admits a surjection to $\gmmath$. This gives a contradiction to Lemma \ref{nosurjectiontogm}.  Invoking Corollary \ref{gminside}, it follows that $$\Phi^{-1}(\mathrm{Z}(\gdrbmath(A)))\cap \gmmath \times \{\mathrm{id}\}=\gmmath \times \{\mathrm{id}\} \xhookrightarrow{} \gmmath \times \mathrm{Z}(\mathrm{Hdg}(A))\xhookrightarrow{}\gmmath\times\mathrm{U}_{E_1}\times\cdots\times\mathrm{U}_{E_{n}}.$$ The same argument applies to the case where $A$ does not have a type IV factor. Finally we obtain that $\gmmath\xhookrightarrow{}\mathrm{Z}(\gdrbmath(A))^{\circ}$. Combining this with the definition of $\Phi$ concludes the proof.
\end{proof}
Let $X$ be a hyper-Kähler variety defined over $\qbar$. Denote by $W:=\mathrm{H}^{2}_{\mathrm{dRB}}(X,\mathbb{Q})$.
\begin{corollary}\label{corgminhk}
Under the above set up, we have $\gdrbmath(W)$ contains $\gmmath$ as the subgroup of homotheties in $\mathrm{GL}(W)$.    
\end{corollary}

\begin{proof}
By the Kuga-Satake correspondence, there exists an abelian variety $A$ defined over $\qbar$ such that $W$ is a direct summand dRB substructure of $V':=\drba\otimes\drba$ (see \cite[Proposition 6.4.3]{andre1996shafarevich} and \cite[Theorem 6.1]{ksv2026rhambetticlassescoefficients}). Thus by Tannakian duality, $\gdrbmath(W)$ is equal to the restriction of the induced action of $\gdrbmath(A)$ on the direct summand $W$ of $V'$. Via $\tilde{\Phi}$, $\gmmath\subset\gdrbmath(A)$ acts on $V'$ and thus on $W$ as $t\mapsto t^{-2}$. Since $\gdrbmath(W)$ is a $\mathbb{Q}$-algebraic subgroup of $\mathrm{GL}(W)$, it necessarily contains the multiplicative group of homotheties.
\end{proof}

\section{Some Corollaries}
Theorem \ref{thmgminside} allows us to define a dRB analogue of the Hodge group. Fix an abelian variety $A/\qbar$ and let $\Phi: \gmmath \times \mathrm{Hdg}(A) \twoheadrightarrow \mathrm{MT}(A)$ be the isogeny from Lemma \ref{hodgetimesgm}.

\begin{definition}\label{gdrbhdefn}
The \textit{de Rham-Betti Hodge group} of $A$, denoted by $\gdrbhmath(A)$, is the image of $\Phi^{-1}(\gdrbmath(A))^\circ$ under the projection map from $\gmmath \times \mathrm{Hdg}(A)$ to $\mathrm{Hdg}(A)$.
\end{definition}
By construction, $\gdrbhmath(A)$ is connected and reductive. Recall from Lemma \ref{gdrbconnectedness} that $\gdrbmath(A)$ is a connected algebraic group as well. Therefore the isogeny $\Phi$ restricts to a surjective homomorphism\begin{equation}\label{eqgmgdrbha}
 \tilde{\Phi}: \gmmath \times \gdrbhmath(A) \twoheadrightarrow \gdrbmath(A)   
\end{equation} sending $(t,g)$ to $ t^{-1}g$. 
\begin{remark}
By Theorem \ref{thmgminside}, the group $\gdrbhmath(A)$ defined above is exactly the group $G_{\mathrm{dRB}}^{1}(A)$ introduced in \cite[p.~21]{ksv2026rhambetticlassescoefficients}, where it is defined as the connected component of the kernel of $\mathrm{det}\colon\gdrbmath(A)\to\gmmath$.
\end{remark}
It is demonstrated in \cite[Lemma 2.8]{bost2016some}  that for any abelian variety $A$ defined over $\qbar$, the dRB structure $\mathbb{Q}_{\mathrm{dRB}}(-1)$ is an object in the Tannakian category $\langle\mathrm{H}_{\mathrm{dRB}}^1(A,\mathbb{Q})\rangle^{\otimes}$. This is because any polarization form on $A$ induces a non-zero morphism of dRB structures \begin{equation}\label{poladrb}
\phi:\drba\otimes\drba\rightarrow\mathbb{Q}_{\mathrm{dRB}}(-1).
\end{equation} 
Recall by Tannakian duality, $\langle\mathrm{H}_{\mathrm{dRB}}^1(A,\mathbb{Q})\rangle^{\otimes}$ is tensor equivalent to the representation category of $\gdrbmath(A)$.
\begin{corollary}\label{gmweights}
Via the group homomorphism (\ref{eqgmgdrbha}), $\gmmath\xhookrightarrow{}\gdrbmath(A)$ acts as $\gmmath \rightarrow \gmmath:t \mapsto t^{-2}$ on $\mathbb{Q}_{\mathrm{dRB}}(-1) \in \mathrm{Ob}\langle\mathrm{H}_{\mathrm{dRB}}^1(A,\mathbb{Q})\rangle^{\otimes}$. 
\end{corollary}
\begin{proof}
The action of $\gmmath$ via $\tilde{\Phi}$ on both sides of the formula (\ref{poladrb}) has to be compatible, which implies the corollary. 
\end{proof}
 
Recall that any character of the Hodge group of an abelian variety is trivial. 
\begin{lemma}\label{chartrivial}
    Let $A$ be an abelian variety defined over $\qbar$. Then any character of $\gdrbhmath(A)$ is trivial. 
\end{lemma}
\begin{proof}
Note that $\mathrm{Z}(\gdrbhmath(A))^{\circ}\subset\mathrm{Z}(\mathrm{Hdg}(A))^{\circ}$. The proof concludes by applying the same argument used for Theorem  \ref{thmgminside}. 
\end{proof}
The following corollary justifies that $\gdrbhmath(A)$ indeed plays the role of the Hodge group in the dRB theory. 
\begin{corollary} \label{gdrbhinv}
Let $M:=\mathrm{H}^1_{\mathrm{dRB}}(A,\mathbb{Q})$ where $A$ is an abelian variety defined over $\qbar$ and let $m$ and $n$ be two non-negative integers such that $m-n$ is an even number. Then
    $$(M^{\otimes m}\otimes M^{*\otimes n})^{\gdrbhmath(A)} \otimes \mathbb{Q}_{\mathrm{dRB}}({\frac{m-n}{2}}) = (M^{\otimes m}\otimes M^{*\otimes n}\otimes \mathbb{Q}_{\mathrm{dRB}}({\frac{m-n}{2}}))^{\gdrbmath(A)}.$$ 
    
\end{corollary}
\begin{proof}
The algebraic group homomorphism $\tilde{\Phi}$ from formula (\ref{eqgmgdrbha}) is surjective on $\qbar$-points. Therefore by the basic invariant theory of algebraic groups, we have that $$(M^{\otimes m}\otimes M^{*\otimes n}\otimes \mathbb{Q}_{\mathrm{dRB}}({\frac{m-n}{2}}))^{\gdrbmath(A)}=(M^{\otimes m}\otimes M^{*\otimes n}\otimes \mathbb{Q}_{\mathrm{dRB}}({\frac{m-n}{2}}))^{\tilde{\Phi}(\gmmath \times \gdrbhmath(A))}.$$ The proof concludes by invoking Corollary \ref{gmweights} and Corollary \ref{chartrivial}.
\end{proof}

\begin{corollary}\label{noodddrbclass}
For any abelian variety $A$ defined over $\qbar$, there is no non-zero dRB class in $\mathrm{H}^{m}_{\mathrm{dRB}}(A,\mathbb{Q}) \otimes \mathbb{Q}_{\mathrm{dRB}}(n)$ if $m \neq 2n$.
\end{corollary}

\begin{proof}
    Note that $\mathrm{H}^{m}_{\mathrm{dRB}}(A,\mathbb{Q})\cong \wedge^{m}\mathrm{H}^{1}_{\mathrm{dRB}}(A,\mathbb{Q})$. Therefore, if $m \neq 2n$, then by Corollary \ref{gmweights}, the set of invariants in $\mathrm{H}^{m}_{\mathrm{dRB}}(A,\mathbb{Q}) \otimes \mathbb{Q}_{\mathrm{dRB}}(n)$ under the action of $\gmmath \xhookrightarrow{}\gdrbmath(A)$ via $\tilde{\Phi}$ is $\{0\}$.
\end{proof}

Note that in the category of all dRB structures, there are uncountably many isomorphism classes of one-dimensional dRB structures. However, recall that in the category of all $\mathbb{Q}$-Hodge structures, a one-dimensional Hodge structure is of the form $\mathbb{Q}(n)$, i.e., $\mathbb{Q}(n)\otimes_{\mathbb{Q}}\mathbb{C}$ has a single summand with bidegree $(-n,-n)$ for some $n\in\mathbb{Z}$. In parallel, we may form the following conjecture regarding one-dimensional dRB structures arising from geometry. 
\begin{conjecture}\label{onedimconj}
For an abelian variety $A/\qbar$, every one-dimensional object in $\langle\mathrm{H}^1_{\mathrm{dRB}}(A,\mathbb{Q})\rangle^{\otimes}$ is of the form $\mathbb{Q}_{\mathrm{dRB}}(k)$ for some $k \in \mathbb{Z}$.
\end{conjecture}

\begin{corollary}
    \label{groupinterpret}
Conjecture \ref{onedimconj} is equivalent to the statement that $\gmmath \cap \gdrbhmath(A) = \{\pm 1\}$.
\end{corollary}
\begin{proof}By formula (\ref{poladrb}), the one-dimensional dRB object $\mathbb{Q}_{\mathrm{dRB}}(-1)\in\mathrm{Ob}\langle \mathrm{H}^1_{\mathrm{dRB}}(A,\mathbb{Q})\rangle^{\otimes}$ induces a character $$\chi_{\phi}:\gdrbmath(A) \rightarrow \gmmath.$$ Lemma \ref{gmweights} states that when restricted to $\gmmath \xhookrightarrow{} \gdrbmath(A)$ via $\tilde{\Phi}$, $\chi_{\phi}$ is an isogeny $\gmmath \rightarrow \gmmath$ of the form $t\mapsto t^{-2}$. Moreover, by Lemma \ref{chartrivial}, we have that $\chi_{\phi}(\gdrbhmath(A))=\{\mathrm{id}\}$. This implies that the order of every element in the algebraic subgroup $\gmmath \cap \gdrbhmath(A)\subset \gmmath$ divides 2. Because every finite algebraic subgroup of $\gmmath$ is cyclic, we have that $\gmmath \cap \gdrbhmath(A)$ either equals $\{\pm1\}\subset\mathrm{GL}(V)$, or equals to $\{1\}$.

Now suppose Conjecture \ref{onedimconj} holds. Then every one-dimensional object in $\langle\mathrm{H}^1_{\mathrm{dRB}}(A,\mathbb{Q})\rangle^{\otimes}$ is of the form $\mathbb{Q}_{\mathrm{dRB}}(k)$, for some $k\in\mathbb{Z}$. Note that $\mathbb{Q}_{\mathrm{dRB}}(k)\cong \mathbb{Q}_{\mathrm{dRB}}(1)^{\otimes k}$. Then by Tannakian duality, every character of $\gdrbmath(A)$, when restricted to $\gmmath\xhookrightarrow{}\gdrbmath(A)$ via $\tilde{\Phi}$ is of the form $t\mapsto t^{2k}$ where $k\in\mathbb{Z}$. Suppose that $\gmmath \cap \gdrbhmath(A)$ is equal to $\{1\}$. Then there exists a character $$\chi_{\frac{1}{2}}:\gdrbmath(A)\rightarrow\gmmath$$ such that its restriction to $\gmmath\xhookrightarrow{}\gdrbmath(A)$ is equal to $t\mapsto t^{-1}$ and $\chi_{\frac{1}{2}}(\gdrbhmath(A))=\{\mathrm{id}\}$. This gives a contradiction. This then implies that $\gmmath \cap \gdrbhmath(A)$ is equal to $\{\pm 1\}\subset \mathrm{GL}(V)$.

On the other hand, suppose $\gmmath \cap \gdrbhmath(A)$ is equal to $\{\pm1\}\subset\mathrm{GL}(V)$. Then given any one-dimensional dRB structure $L_{\alpha} \in\mathrm{Ob}\langle \mathrm{H}^1_{\mathrm{dRB}}(A,\mathbb{Q})\rangle^{\otimes}$, the induced character $\chi_{\alpha}:\gdrbmath(A)\rightarrow\gmmath$ satisfies $\chi_{\alpha}(\gdrbhmath(A))=\{\mathrm{id}\}$ by Lemma \ref{chartrivial} and restricts to $\gmmath \xhookrightarrow{} \gdrbmath(A)$ via $\tilde{\Phi}$ as $t\mapsto t^{m}$. By the assumption that $\gmmath \cap \gdrbhmath(A)$ is equal to $\{\pm1\}\subset\mathrm{GL}(V)$, we deduce that $m$ is an even number. Therefore by the Tannakian duality, we have that $L_{\alpha}$ is isomorphic to $\mathbb{Q}_{\mathrm{dRB}}(\frac{m}{2})$ as dRB structures.
\end{proof}
\begin{remark}
The proof of the above corollary demonstrates that $\sqrt{2\pi\sqrt{-1}}$ is $\textit{not}$ a period for an abelian variety $A$ defined over $\qbar$ if and only if $\gmmath\cap\gdrbhmath(A)=\{\pm 1\}$.
\end{remark}

\printbibliography[
heading=bibintoc,
title={References}
]



\end{document}